\documentclass[12pt]{amsart}
\usepackage{amsthm,amsopn,amsfonts,amssymb,amsmath}%,pb-diagram}
\usepackage{hyperref}
\usepackage{euler, amsfonts, amssymb, latexsym, epsfig}%,epic}

\font\co=lcircle10
\def\boxcross{\ \smash{\lower6.5pt\hbox{\rlap{\hskip4.5pt\vrule height13.5pt}}
                \raise0pt\hbox{\rlap{\hskip-2pt \vrule height.4pt depth0pt
                        width13.5pt}}}\hskip12.7pt}

\def\boxelbow{\ \hskip.1pt\smash{%
               \hbox{\co \hskip 5.5pt\rlap{\mathsurround=0pt\rlap{\mathsurround=0pt\char'006}\lower0.4pt\rlap{\char'004}}
                \lower6.5pt\rlap{\hskip-0.2pt\vrule height3pt}
                \raise3.5pt\rlap{\hskip-0.2pt\vrule height3.2pt}}
                \hbox{%
                  \rlap{\hskip-6.4pt \vrule height.4pt depth0pt
width2.5pt}%
                  \rlap{\hskip4.05pt \vrule height.4pt depth0pt
width3.1pt}}}
                \hskip8.7pt}

\newtheorem{theorem}{Theorem}[section]
\newtheorem{Theorem}[theorem]{Theorem}
\newtheorem*{thm*}{Theorem}
\newtheorem{Lemma}[theorem]{Lemma}
\newtheorem{Proposition}[theorem]{Proposition}
\newtheorem{Corollary}[theorem]{Corollary}

\newtheorem*{Conjecture*}{Conjecture}

\theoremstyle{definition}

\theoremstyle{remark}

\newcommand\onto{\twoheadrightarrow}

\def\acts{\hspace{.1cm}{
    \setlength{\unitlength}{.25mm}
    \linethickness{.12mm}
    \begin{picture}(8,8)(0,0)
    \qbezier(7,6)(4.5,8.3)(2,7)
    \qbezier(2,7)(-1.5,4)(2,1)
    \qbezier(2,1)(4.5,-.3)(7,2)
    \qbezier(7,6)(6.1,7.5)(6.8,9)
    \qbezier(7,6)(5,6.1)(4.2,4.4)
    \end{picture}\hspace{.1cm}
    }}

\newcommand\defn[1]{{\bf #1}}

\newcommand\calO{{\mathcal O}}
\newcommand\calD{{\mathcal D}}
\newcommand\calF{{\mathcal F}}

\newcommand\into{\hookrightarrow}

\newcommand\calL{{\mathcal L}}
\newcommand\lie[1]{{\mathfrak #1}}

\newcommand\naturals{{\mathbb N}}
\newcommand\reals{{\mathbb R}}
\newcommand\complexes{{\mathbb C}}
\newcommand\integers{{\mathbb Z}}

\newcommand\tensor{\otimes}

\newcommand\iso{\mathrel{\cong}}
\newcommand\junk[1]{}

\newcommand\codim{{\mathop{\mathrm{codim}}}}

\begin{document}

\title[Positive formul\ae\ for $SL_3(\reals)$ $K$-types 
  and a smooth $K$-orbit closure Blattner formula]
{Positive formul\ae\ for $K$-types of $SL_3(\reals)$-irreps \\
  \ and a Blattner formula for smooth $K$-orbit closures}
\author{Allen Knutson}
\date{September 30, 2014}
\dedicatory{To David Vogan for his $60$th birthday}

\begin{abstract}
  We prove a version of Blattner's conjecture, 
  for irreducible subquotients of principal series representations
  with integral infinitesimal character
  of a real reductive Lie group whose Beilinson-Bernstein 
  $\calD$-module is supported on a $K$-orbit with smooth closure.
  (The cases usually considered are closed orbits, or their
  preimages along $G/B_G^- \to G/P^-$.)
%  Even in the discrete series case this formula refines Blattner's.
  We apply this to $G_\reals = SL_3(\reals)$, where all four
  $K$-orbits on $G/B_G^-$ have smooth closure, and refine the resulting
  alternating-sum formul\ae\ to ones with only positive terms.
\end{abstract}

\maketitle

\setcounter{tocdepth}{1}
%{\footnotesize \tableofcontents}

%\section{Introduction, and statement of results}

\section{Statement of results}

Let $G_\reals = SL_3(\reals)$, with maximal compact $K_\reals = SO(3,\reals)$
preserving the form with Gram matrix $\left[{}_{{}_1} {}^{{}_1} {}^{{}^1}\right]$,
and flag manifold $G/B_G^- \iso \{ (0 < V_1 < V_2 < \complexes^3) \}$. 
Let $rank(V)$ denote the rank on $V \leq \complexes^3$ of the restriction of 
the $K$-invariant orthogonal form on $\complexes^3$.
Then $K$ preserves the smooth subvarieties 
$$ \calO_1 := \{ (V_1 < V_2) : rank(V_1) = 0 \},\quad
   \calO_2 := \{ (V_1 < V_2) : rank(V_2) = 1 \} $$
and the full set of $K$-orbit closures is 
$G/B_G^-,\ \calO_1,\ \calO_2,\ \calO_1 \cap \calO_2$, all smooth.

In \S \ref{sec:Blattner} we prove a Blattner-type formula for the
$K$-multiplicities in $(\lie{g},K)$-irreps 
with integral infinitesimal character whose
corresponding $\calD$-modules are supported on smooth $K$-orbit closures.
Since the standard modules for the open orbit on $SL(3)/B^-$ 
have composition series whose factors are the irreps for the four orbits,
we could hope to define a set $C$ that computes the $SO(3)$-multiplicities
of the standard module and
breaks as $C =\coprod_\calO C_\calO$ giving formul\ae\ for those irreps,
and this is our main result (proved in \S \ref{sec:proofofmain}).

\begin{Theorem}\label{thm:main}
  Fix an infinitesimal character $(a,b) + \rho_G \in \naturals^2$
  of $G_\reals=SL_3(\reals)$.
  Define the set
  $$ C := \left\{(c,d) \in \naturals^2 \ :\ c \equiv a\bmod 2,\
  \text{and if }c=0 \text{ then } d \equiv b\bmod 2 \right\}
  $$
  and the projection $\pi: C\to \naturals$, $(c,d)\mapsto c+d$.
  This $C$ decomposes into four regions:
  $$    C_{G/B_G^-} := \{ (c,d) \in C\ : c\leq a,\ d\leq b \} $$
  $$   C_{\calO_1} := \{ (c,d) \in C\ : c > a,\ d\leq b \} \qquad\qquad
  C_{\calO_2} := \{ (c,d) \in C\ : c \leq a,\ d > b \}  $$
  $$    C_{\calO_1\cap \calO_2} := \{ (c,d) \in C\ : c > a,\ d > b \}   $$
  The $(\lie{sl}_3,SO(3))$-irrep with infinitesimal character $(a,b) + \rho_G$,
  supported on an orbit closure $\calO$, contains the $SO(3)$-irrep of
  dimension $2n+1$ with multiplicity $\#\left( \pi^{-1}(n) \cap C_\calO \right)$.
\end{Theorem}

\begin{figure}[htbp]
  \centering
  \epsfig{file=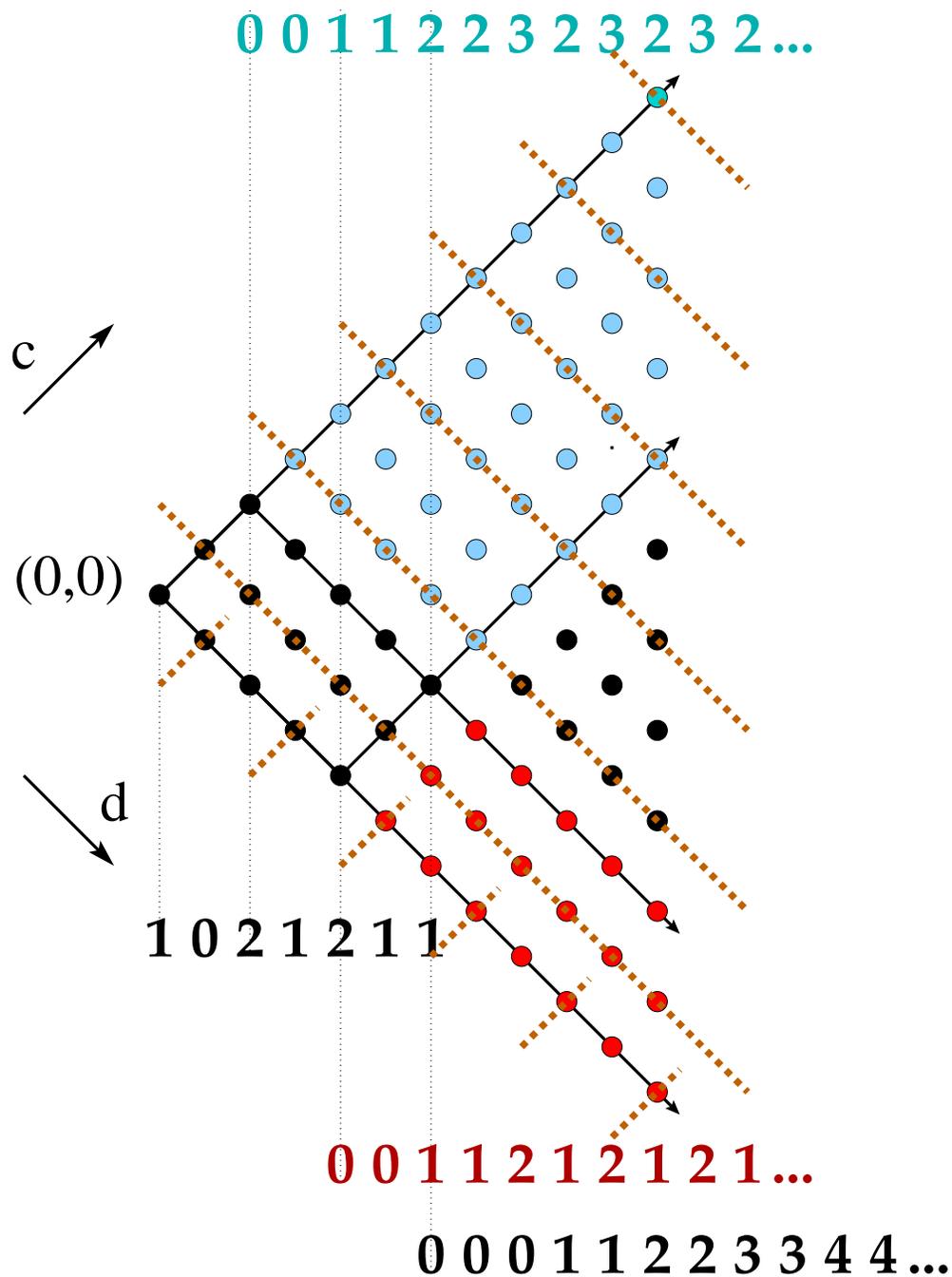,height=7in}
  \caption{The black lines denote the boundaries of the four regions
    in $C$. Each alternate line of constant $c$ is crossed out by
    brown dashed lines, and when $c=0$, every other dot is crossed out.
    When we project the remaining dots vertically to $\naturals$, 
    the four regions give the multiplicities pictured (counted from the
    left corners of the regions). The vertical dotted lines go through
    the corners of $C_{SL_3/B}$ and bound (up to $\rho$-shifts) the regions 
    in the $K$-multiplicities where the piecewise-quasipolynomial changes.}
  \label{fig:Cpic}
\end{figure}

See figure \ref{fig:Cpic} for a picture of $C$ and its four pieces in
the case $(a,b)=(2,4)$, and the multiplicities induced by their
projections to $\naturals$.  It would be interesting to relate the
decomposition of $C$ to the gluing of conormal varieties in \cite{MO}.
Other combinatorial investigations of the Blattner formula 
appear in \cite{WZ08,H}.

\section*{Acknowledgments}

This work was inspired by David Vogan, from whom I've learned so much --
but not the yoga of tempered representations and $\calD$-modules,
which I learned from Peter Trapa, without whom this
paper couldn't have come into existence. (When Peter is unavailable,
a nice paper reference is \cite{M}.)
Of course, one could use a transitivity argument to again credit David.
I also thank Wilfried Schmid for useful discussion about closed orbits
in general and the closed $SO(3)$-orbit in particular.

I thank Birgit Speh for explaining to me that in the long history of
this subject, pictures such as those in figure \ref{fig:Cpic} to compute 
$K$-multiplicities were essentially unpublishable due to technological 
restrictions, but were common on David's blackboard. 
They {\em may} have appeared in her thesis (on $SL_3(\reals)$-representations),
which (unlike \cite{H}) is unfortunately not available on the MIT website.

Roger Howe pointed me to his paper \cite{Howe} which also computes the
$K$-types. Rather than using the hamfisted character-theoretic approach
we take here, he looks much more deeply into operators on the representation
itself, finding a basis for each highest weight space.

\section{A Blattner-type formula for smooth orbit closures}
\label{sec:Blattner}

Blattner's conjecture concerns $K$-weight multiplicities, which we
will derive from $T_K$-weight multiplicities, and those from
$T_K$-equivariant localization in $K$-theory (a slightly nontrivial step,
since localization gives us rational functions which carry slightly less
information than the relevant power series). There is an extra complication,
which is that the $T_K$-weight multiplicities are infinite, and we need
to split them using an extra action of $\complexes^\times$.

\subsection{Generalities about equivariant localization 
  and multiplicity functions}

\begin{Proposition}\label{prop:contribution}
  Let $X \subseteq Y$ be a pair of complex varieties, 
  $Y$ smooth and $X$ projective, carrying compatible actions of a torus $T$,
  with $Y^T$ finite.
  Let $D_X$ be the sheaf of distributions on $Y$ supported on $X$ 
  (so, a module over the filtered sheaf $\calD_Y$ of differential operators),
  and $\calL$ a $T$-equivariant line bundle on $Y$.

  Pick generators for $D_X$ as a $\calD_Y$-module, so we can 
  define an associated graded $gr\ D_X$, 
  which gives a $(T\times \complexes^\times)$-equivariant 
  coherent sheaf on ${\bf Spec}(gr\ \calD_Y) = T^*Y$; here 
  the $\complexes^\times$ acts by {\em inverse} rescaling action the fibers 
  of $T^* Y$ (the inverse to make its character $z$ and not $z^{-1}$ below).
  Extend $\calL$ to an equivariant line bundle on $T^* Y$ by pulling it back 
  along the projection $T^* Y \onto Y$. 
  Then each $H^i(T^* Y;\ gr\ D_X \otimes \calL)$
  has finite-dimensional $\complexes^\times$-weight spaces, 
  hence finite-dimensional $(T\times \complexes^\times)$-weight spaces,
  so its character is a well-defined formal power series.
  (It does not usually have finite-dimensional $T$-weight spaces.)

  The $(T\times \complexes^\times)$-character on 
  $\chi(T^* Y;\ gr\ D_X \otimes \calL) 
  := \sum_i (-1)^i\ H^i(T^* Y;\ gr\ D_X \otimes \calL)$,
  the sheaf Euler characteristic,
  can be computed by equivariant localization 
  (thereby showing that its power series sums to a rational function).
  If $f \in X^T$ is a smooth point of $X$, then the contribution
  to the character from the point $f$ is
  $$ 
  t^{wt(\calL|_f)}
  \prod_{\mu \in wts(T_f X)} \frac{1}{1 - t^{-\mu}}
  \prod_{\mu \in wts(T_f Y/T_f X)} \frac{t^{\mu} z}{1 - t^\mu z}.
  $$
\end{Proposition}

\begin{proof}
  For any coherent $(T \times \complexes^\times)$-equivariant sheaf $\calF$
  on $T^* Y$, the localization formula from \cite{T} says
  we can localize the character around the fixed points 
  $(T^* Y)^{T_K \times \complexes^\times} = Y^{T_K}$,
  \begin{eqnarray*}
     Tr((t,z)|_{\chi(T^*Y;\ \calF)}) 
     &=& \sum_{f \in Y^{T_K}}
     \frac{L_f}{\prod_{\mu \in wts(T_f Y)} (1 - t^{-\mu}) (1 - t^\mu z)},
     \qquad (t,z) \in T_K \times \complexes^\times \\
  \end{eqnarray*}
  where the two factors in the product in the denominator come from the tangent
  and cotangent lines in $T_f(T^* Y)$, and the $L_f$ are 
  regular functions on $T\times \complexes^\times$ derived from $\calF$.
  (More specifically, we can compute $L_f$ by picking an 
  equivariant projective resolution of $\calF$ and alternating-sum 
  the characters of the fibers of the resulting vector bundles.)
  In the case at hand, where $\calF = gr\ D_X \otimes \pi^*(\calL)$
  and $D_X$ was supported along $X$,
  the contribution $L_f$ vanishes unless $f\in X$ (and hence $X^{T_K}$).
  
  When $f$ is a {\em smooth} point of $X$, we can reduce to
  the local situation 
  $T \acts (V_1 \subset V_2)$ where $V_1,V_2$ are vector spaces. There
  $$ Tr\left((t,z)|_{\chi(T^* V_2;\ gr\ D_{V_1})}\right) 
  = \prod_{\mu \in wts(V_1)} \frac{1}{1 - t^{-\mu}}
  \prod_{\mu \in wts(V_2/V_1)} \frac{t^{\mu} z}{1 - t^\mu z}
  $$
  which one derives by observing that $gr\ D_{V_1}$ is a free sheaf
  of rank $1$ on the conormal variety $V_1 \times (V_2/V_1)^*$, 
  with generator of $T$-weight $\sum_{\mu \in wts(V_2/V_1)} \mu$.
  Letting $V_1 = T_f X$ and $V_2 = T_f Y$, we get the desired formula.
\end{proof}

We actually want to study this representation
$\chi(T^* Y;\ gr\ D_X \otimes \calL)$ by its power series,
or weight multiplicity function (an $\naturals$-valued
function on the weight lattice), not by its character
(a rational function on $T\times \complexes^\times$)
provided by equivariant localization.
The basic problem in going backwards from the character is that the
two power series $1+x+x^2+\ldots$ and $-x^{-1}-x^{-2}-\ldots$ are not equal, 
but their rational functions $\frac{1}{1-x} = \frac{-x^{-1}}{1-x^{-1}}$ are. 
So how to choose a power series given a rational function,
especially one presented as a sum of rational functions with even
worse denominators?

\begin{Lemma}\label{lem:recoverP}
  Say $P$ is a multivariable power series whose exponent vectors live in a cone
  on which some generic coweight $\tau$ is proper and bounded above.
  Assume $P$ adds up to a rational function $r$ of the form $\sum_f r_f$ 
  where each $r_f$'s denominator is of the form $\prod (1-t^{-\beta})$.  
  If we use the identity
  $\frac{1}{1-x} = \frac{-x^{-1}}{1-x^{-1}}$ to flip the denominators in
  each $r_f$ to make each $\langle \tau, \beta \rangle > 0$, and use
  that choice to thereby choose a power series $P_f$ for each $f$, 
  then $P = \sum_f P_f$ as power series (refining the fact that
  $r = \sum_f r_f$ as rational functions).
\end{Lemma}

\begin{proof}
  In the one-variable case, this just says that a rational function
  (with this sort of denominator) has a unique Laurent series expansion. 
  In a quite different situation, two different finitely supported
  multivariable Laurent ``series'' can't have the same associated
  rational function (the Laurent polynomial).

  Now the general case.
  Since the support is in a proper cone we can perturb $\tau$ to be rational, 
  then scale it to be integral. Thinking of $P$ by its coefficients
  so as a function on $\integers^d$, we can use $\langle \tau, \bullet\rangle$
  to marginalize $P$ to a power series $P'$ in one variable. By the
  first statement above, that $P'$ is unique. Then over each exponent of $P'$
  we have finitely many terms in $P$, and can adapt the second statement above.
\end{proof}

This question is often phrased as one of regularizing the Fourier transform 
to be able to apply to rational functions.
In the cases typically considered (e.g. in \cite{GLS}) the ``series'' $P$ is
finitely supported and $\tau$ is therefore arbitrary (once generic).
In the most familiar example, $P$ is the weight multiplicity function of a 
finite-dimensional $G$-irrep $V_\lambda$ and localization gives the manifestly 
Weyl-symmetric formula\footnote{We record here the conventions that lead to 
  this formula in this form. The positive roots $\Delta_+^G$ are those in 
  the Borel subgroup $B_G$. Around the basepoint of $G/B_G^-$, the isotropy 
  weights are $\Delta_-^G$, so the ring of functions on that tangent space 
  or big cell has the character $1/Den(G)$, which we then twist by $t^\lambda$.
  Then the other summands are all Weyl translates of this one.
  When flipping to the standard Weyl character formula, we want
  each of the $1-t^{-\beta}$ factors in the denominator to have 
  $\beta \in \Delta_+^G$.
  }
$\sum_{w\in W_G} w\cdot \frac{t^\lambda}{Den(G)}$, 
where $Den(G) = \prod_{\beta\in \Delta^G_+} (1-t^{-\beta})$ is the Weyl denominator
and $\Delta^G_+$ is determined from $\tau$.
In this case all denominators flip to the same $Den(G)$, giving the 
Weyl character formula and its Fourier transform the Kostant
multiplicity formula. But when $P$ is noncompactly supported we
will need to choose $\tau$ more carefully.

Nevertheless, start with $\tau$ a generic coweight of $T$,
and a finite-dimensional $T$-vector space $V$ with $V^T = 0$.
If all of $V$'s weights $wts(V)$ lie in a proper half-space
of $T$'s weight lattice,
then the $T$-weight multiplicities on the ring of functions on $V$ are finite, 
and its character is the rational function 
$\prod_{\mu\in wts(V)} \frac{1}{1-t^{-\mu}}$.
Define
\begin{eqnarray*}
  down(V)  &:=& \{ \mu \in wts(V) \ : \ \langle \tau, \mu\rangle < 0 \} \\
   ind(V)  &:=& \#down(V) \\
  \rho(V)  &:=& \sum down(V) \qquad
  \text{so $\langle \tau,\rho(V) \rangle < 0$, unless $down(V)=\emptyset$}\\
  wts(V)_+ &:=& \{ \mu\ \mathop{sign}(\langle \tau, \mu\rangle)\ :\ 
     \mu \in wts(V) \}
\end{eqnarray*}
where $wts(V)$ denotes the set-with-multiplicity of $T$-weights.
Call these together the \defn{$\tau$-positivity setup,} and use them to rewrite
the character of the ring of functions on $V$
\begin{eqnarray*}
  \prod_{\mu\in wts(V)} \frac{1}{1-t^{-\mu}}
  &=& 
  \prod_{\mu\in wts(V),\ \langle \mu,\tau \rangle < 0} \frac{1}{1-t^{-\mu}} 
  \prod_{\mu\in wts(V),\ \langle \mu,\tau \rangle > 0} \frac{1}{1-t^{-\mu}} \\
  &=& 
  \prod_{\mu\in wts(V),\ \langle \mu,\tau \rangle < 0} \frac{-t^\mu}{1-t^{\mu}}
  \prod_{\mu\in wts(V),\ \langle \mu,\tau \rangle > 0} \frac{1}{1-t^{-\mu}} \\
  &=& 
  (-1)^{ind(V)}\ t^{\rho(V)}
  \prod_{\mu\in wts(V),\ \langle \mu,\tau \rangle < 0} \frac{1}{1-t^{\mu}}
  \prod_{\mu\in wts(V),\ \langle \mu,\tau \rangle > 0} \frac{1}{1-t^{-\mu}} \\
  &=&
  (-1)^{ind(V)}\ t^{\rho(V)}
  \prod_{\mu\in wts_+(V)} \frac{1}{1-t^{-\mu}}.
\end{eqnarray*}
Let 
$$ \kappa_m(\vec v; S) := \# \left\{ \text{ways to write $\vec v$ as a 
  sum of exactly $m$ vectors from $S$, with multiplicity} \right\} $$
be a sort of graded vector partition function on a set-with-multiplicity $S$,
with
$$ \kappa(\vec v; S) = \sum_{m\in \naturals} \kappa_m(\vec v; S) $$
the usual one. 

\begin{Proposition}\label{prop:FT}
  From the $T$-coweight $\tau$ build a $(T\times \complexes^\times)$-coweight
  $(\tau, N \ll 0)$. With this choice, we can Fourier transform the term 
  from proposition \ref{prop:contribution} 
  to the virtual weight multiplicity function
  \begin{eqnarray*}
    (\lambda,d) \ \mapsto\ (-1)^{ind(T_f X)} \ 
    &&\kappa\bigg( (\lambda - wt(\calL|_f) - \rho(T_f X)
    - \sum_{\mu \in wts(T_f Y/T_f X)} \mu,\ d - \codim_Y X);\\
      && \{ (-\mu,0) : \mu\in wts_+(T_f X) \} \coprod
      \{ (\mu,1) : \mu\in wts(T_f Y/T_f X) \}     \bigg)
  \end{eqnarray*}
  and the $(T\times \complexes^\times)$-weights occurring in 
  $\chi(T^* Y;\ \calL \tensor gr\ D_X)$ lie in a cone
  on which the coweight $(\tau,N)$ of $T\times\complexes^\times$
  is proper and bounded above.
\end{Proposition}

\begin{proof}
  The benefit of taking $N\ll 0$ in the coweight $(\tau,N)$ 
  is that the $z$ factors in the second product ensure that
  none of them need be flipped to satisfy lemma \ref{lem:recoverP}.
  The first product we rewrite as explained above:
  \begin{eqnarray*}
  \prod_{\mu \in wts(T_f X)} \frac{1}{1 - t^{-\mu}}
  &=& (-1)^{ind(T_f X)}\ t^{\rho(T_f X)} 
  \prod_{\mu \in wts_+(T_f X)} \frac{1}{1 - t^{-\mu}}
  \end{eqnarray*}
  so the rational function contributed by proposition \ref{prop:contribution} 
  can be rewritten (with $(\tau,N)$-positive denominators) as
  \begin{eqnarray*}
    &&
    (-1)^{ind(T_f X)}\ t^{wt(\calL|_f) + \rho(T_f X)}
    \prod_{\mu \in wts_+(T_f X)} \frac{1}{1 - t^{-\mu}}
    \prod_{\mu \in wts(T_f Y/T_f X)} \frac{t^{\mu} z}{1 - t^\mu z} \\
    &=&
    (-1)^{ind(T_f X)}\ t^{wt(\calL|_f) + \rho(T_f X) + \sum_{\mu \in wts(T_f Y/T_f X)} \mu}
    \ z^{\codim_Y X}
    \prod_{\mu \in wts_+(T_f X)} \frac{1}{1 - t^{-\mu}}
    \prod_{\mu \in wts(T_f Y/T_f X)} \frac{1}{1 - t^\mu z} 
  \end{eqnarray*}
  whose imputed Fourier transform is the formula claimed.
\end{proof}

Note that if we try to forget the $\complexes^\times$-action in the
above, we run into a vector partition function on the set
$ \{ -\mu : \mu\in wts_+(T_f X) \} \ \coprod\ wts(T_f Y/T_f X) $,
which may very well contain a vector and its negative,
making the vector partition function infinite. 
The retention of the $\complexes^\times$ therefore looks crucial,
but we will sometimes be able to get rid of it,
as in the end of theorem \ref{thm:Blattnerish}.

\subsection{A $K$-multiplicity formula for certain $\calD$-modules}

Some version of the following argument seems to be well-known to the
experts (see e.g.  \cite[p110--111]{BS}), at least for proving
Blattner's original conjecture, but we need it in greater generality
than we could find in the literature.

\begin{Theorem}\label{thm:geometry}
  Let $X \subseteq Y$ be a pair of smooth complex projective varieties,
  carrying compatible actions of a connected complex reductive group $K$. 
  Assume that $K$'s maximal torus $T_K$ acts on $Y$ (and hence $X$) 
  with isolated fixed points, and that all $K$-stabilizers are solvable.
  Fix a generic real dominant coweight $\tau$ of $K$,
  with which to define the positive roots $\Delta^K_+$ of $K$
  and a $\tau$-positivity setup on each $T_f X$. 
  Then $wts(T_f X)_+ \supseteq \Delta_+^K$ for each $f$.

  Let $D_X$ be the sheaf of distributions on $Y$ supported on $X$, and
  $\calL$ a $K$-equivariant line bundle on $Y$.  Then for $\lambda$ a
  dominant $K$-weight, the multiplicity of the finite-dimensional
  $K$-irrep $V_\lambda$ in the $\complexes^\times$-weight $d$ part of
  $\chi(T^* Y;\ gr\ D_X \otimes \calL)$ is
  \begin{eqnarray*}
    &&  \sum_{f\in X^T}
    (-1)^{ind(T_f X)} 
    \kappa\bigg(
     \big(\lambda -wt(\calL|_f) -\rho(T_f X) -\sum_{wts(T_f Y / T_f X)} \mu,\quad
     d - \codim_Y X\big)\ ; \\ 
   && \qquad\qquad\qquad
   \{ (-\mu,0) : \mu\in wts_+(T_f X) \setminus \Delta^K_+ \} \ \coprod\
      \{ (\mu,1) : \mu\in wts(T_f Y/T_f X) \}     \bigg).
  \end{eqnarray*}
  This is $0$ for $d < \codim_Y X$, and nonnegative for $\calL$ ample enough.
\end{Theorem}

Sanity check (Borel-Weil): $X = Y = K/B_K^-$, 
with $\calL$ the $\mu$ Borel-Weil line bundle (weight $\mu$ over the basepoint).
Then $T_f Y/T_f X = 0$, and each $wts_+(T_fX) = \Delta^K_+$, 
so the vector partition function becomes a Kronecker delta. 
Also, $X^T \iso W_K$, with $ind(T_w X) = \ell(w)$, 
$wt(\calL|_w) = w \cdot \lambda$, and $\rho(T_w X) = w\cdot\rho_K - \rho_K$.
The result (summed over $d$, vanishing unless $d=0$) is
\begin{eqnarray*}
  \sum_{d\in\naturals} \sum_{w\in W_K} (-1)^{\ell(w)} \kappa\left(
    (\lambda - w\cdot \mu + \rho_K - w\cdot \rho_K, d); \ \emptyset\right)
  &=& \sum_{w\in W_K} (-1)^{\ell(w)}\ [\lambda + \rho_K = w\cdot(\mu + \rho_K)] \\
  &=& [\lambda + \rho_K = \mu + \rho_K] \qquad = [\lambda = \mu] 
\end{eqnarray*}
as expected.

\begin{proof}
  Now that $X$ is smooth, the localization formula simplifies to
  proposition \ref{prop:contribution} at {\em every} $f\in X^T$,
  and by lemma \ref{lem:recoverP} we obtain the 
  $(T_K\times\complexes^\times)$-multiplicities when we sum up the
  virtual multiplicity functions from proposition \ref{prop:FT}:
  \begin{eqnarray*}
    \sum_{f\in X^T} (-1)^{ind(T_f X)} \ 
    &&\kappa\bigg( (\lambda - wt(\calL|_f) - \rho(T_f X)
    - \sum_{\mu \in wts(T_f Y/T_f X)} \mu,\ d - \codim_Y X);\\
      && \{ (-\mu,0) : \mu\in wts_+(T_f X) \} \ \coprod\
      \{ (\mu,1) : \mu\in wts(T_f Y/T_f X) \}     \bigg).
  \end{eqnarray*}
  
  Note that $\lie{k}/\lie{stab_{\lie{k}}(x)} \into T_f X$
  as $T_K$-representations, so by the assumption of solvable stabilizers, 
  $wts(T_f X)$ must contain at least one of $\beta,-\beta$ 
  for each positive root $\beta\in \Delta^K_+$.
  Hence $wts(T_f X)_+ \supseteq \Delta^K_+$.

  To go from the $(T_K\times \complexes^\times)$-multiplicities
  to the $(K\times \complexes^\times)$-multiplicities,
  we have to undo the Kostant multiplicity formula, 
  which amounts to applying difference operators in the directions
  of $K$'s positive roots. The effect is to remove those vectors
  from the input to the vector partition function, giving the claimed formula.

  On the positive Weyl chamber, that formula computes
  the $(K \times \complexes^\times)$-multplicities
  in $\chi(T^* Y;\ gr\ D_X \tensor \calL)$. (Outside the
  positive Weyl chamber it is Weyl-antisymmetric around $-\rho_K$,
  secretly computing the Euler characteristic of the $\lie{n}_K$-cohomology.)
  When $\calL$ is ample enough, this $\chi = \dim H^0$, hence is nonnegative.  
\end{proof}

\subsection{Closed $K$-orbits}

Hereafter $G$ is a complex connected reductive group, with $K$ the
identity component of the fixed points of a holomorphic involution $\theta$.
The ambient space\footnote{As indicated in the previous footnote, this
  convention leads to the cleanest localization formula.}
$Y = G/B_G^-$ is a flag manifold, 
and $X\subseteq Y$ is a $K$-orbit closure on it.
Pick $T_K \leq K$ then $T_G \geq T_K$, so we have a restriction map
$ T_G^* \onto T_K^* $
that we won't require be an isomorphism.
The assumptions on fixed points and stabilizers hold in this case, 
but a $K$-orbit closure $X$ will usually be singular. 
If $X$ is a {\em closed} orbit then it is smooth (but not vice versa).

To avoid cluttering the notation (and since we essentially never use 
$T_G$-weights, only $T_K$-weights) we won't write in the restriction map;
when mixing $T_G$-weights and $T_K$-weights 
(e.g. in $\Delta_+^G \setminus \Delta_+^K$) assume that the restriction
map has been applied to the $T_G$-weights.
Of course, this may mean that $\Delta_+^G \setminus \Delta_+^K$
refers to a set-with-multiplicity.

\begin{Theorem}\label{thm:Blattnerish}
  If $X$ is a closed $K$-orbit on $G/B_G^-$, and $\calL$ is the %Borel-Weil 
  line bundle on $G/B_G^-$ corresponding to a dominant $G$-weight $\nu$,
  then the multiplicity of the $K$-irrep $V_\lambda$ in
  $\chi(G/B_G^-;\ D_X \otimes \calL)$ is
  \begin{eqnarray*}
    && \sum_{w\in W_K}     (-1)^{\ell(w)}  \kappa\bigg(
    (\lambda +\rho_K) - w\cdot\big(\nu + 2\rho_G - \rho_K\big)
    \ ; \    w\cdot( \Delta^G_+ \setminus \Delta^K_+)
    \bigg) \\
    &=&
    \sum_{w\in W_K}     (-1)^{\ell(w)}  \kappa\bigg(
    w\cdot(\lambda +\rho_K) - \big(\nu + 2\rho_G - \rho_K\big)
    \ ; \     \Delta^G_+ \setminus \Delta^K_+
    \bigg).
  \end{eqnarray*}
  For $\nu$ sufficiently dominant, 
  if we refine $\kappa()$ to $\sum_d \kappa_d()$,
  then each $\kappa_d()$ is individually positive.
\end{Theorem}

\junk{
Sanity check (Borel-Weil) redux: let $K=G$. Then this becomes
\begin{eqnarray*}
  \sum_{w\in W_G} (-1)^{\ell(w)} \kappa\left(
    w\cdot(\lambda +\rho_G) - \left(\nu + \rho_G\right)
    \ ;\ \emptyset \right)
  &=&  [\lambda + \rho_G = \nu + \rho_G]    
\end{eqnarray*}
}

In the equal-rank case, this is Blattner's conjecture
for integral infinitesimal character.
It is off by a $\rho_G$ from the usual formula because it is stated in
terms of the line bundle rather than the infinitesimal character.

\junk{

We cite the Blattner formula (which requires $T_K = T_G$) for comparison:
$$ \sum_{w\in W_K} (-1)^{\ell(w)}
\kappa\bigg( w\cdot(\lambda + \rho_K) -(\nu+\rho_G- \rho_K)\ ; \
\Delta^G_+ \setminus \Delta^K_+ \bigg).
$$
and check it for $G=K$:
\begin{eqnarray*}
  \sum_{w\in W_K} (-1)^{\ell(w)}
  \kappa\bigg( w\cdot\lambda -\nu + w\cdot\rho_G\ ; \
  \emptyset \bigg) 
  &=& \sum_{w\in W_G} (-1)^{\ell(w)} [w\cdot(\lambda+\rho_G) = \nu] \\
  &=& [\lambda+\rho_G = \nu] \\
\end{eqnarray*}

}

\begin{proof}
  For this we take $\tau$ a generic dominant coweight of $G$,
  which by restriction (and our choice $T_G \geq T_K$) gives a
  coweight of $K$ as well, defining $\Delta_+^G$ and $\Delta_+^G$.
  The $wt$ and $wts$ notation hereafter refer to $T_K$-weights.

  For this pair $(X,Y)$ every $wts(T_f X)_+ = \Delta^K_+$,
  so theorem \ref{thm:geometry} becomes
  \begin{eqnarray*}
    && \sum_{f\in X^T} (-1)^{ind(T_f X)} \kappa\bigg(
    \big(\lambda -wt(\calL|_f) -\rho(T_f X) -\sum_{wts(T_f Y / T_f X)} \mu,\ 
    d -\codim_Y X\big);
    \quad wts(T_f Y/T_f X) \times \{1\}     \bigg) \\
    &=& \sum_{f\in X^T} (-1)^{ind(T_f X)} \kappa_{d -\codim_Y X}\bigg(
    \lambda -wt(\calL|_f) -\rho(T_f X) -\sum_{wts(T_f Y / T_f X)} \mu;
    \quad wts(T_f Y/T_f X)      \bigg).
  \end{eqnarray*}
  Now notice that
  \begin{itemize}
  \item $X^T = W_K \cdot B_G^-/B_G^-,$ so take $f = w B_G^-/B_G^-$ hereafter;
  \item $ind(T_f X) = \ell(w)$ for $w\in W_K$,
  \item $\rho(T_f X) = w\cdot\rho_K - \rho_K$, and
  \item $\sum_{wts(T_f Y / T_f X)} \mu = w\cdot 2(\rho_G - \rho_K)$
    as $T_K$-weights,
  \end{itemize}
  so we can rewrite $\kappa$'s first argument thusly
  \begin{eqnarray*}
    \lambda -wt(\calL|_f) -\rho(T_f X) -\sum_{wts(T_f Y / T_f X)} \mu
    &=& \lambda -w\cdot \nu-(w\cdot\rho_K -\rho_K) -2w\cdot(\rho_G - \rho_K)\\
    &=& (\lambda +\rho_K) -w\cdot \left(\nu + 2\rho_G - \rho_K\right)
  \end{eqnarray*}
  and the $m$th term in the sum becomes
  \begin{eqnarray*}
    && \sum_{w\in W_K}     (-1)^{\ell(w)} \kappa_{d-|\Delta_-^G \setminus \Delta_-^K|}\bigg(
    (\lambda +\rho_K) -w\cdot \left(\nu + 2\rho_G - \rho_K\right)
    \ ; \     w \cdot (\Delta^G_+ \setminus \Delta^K_+)
    \bigg) \\
    &=& \sum_{w\in W_K}     (-1)^{\ell(w)} \kappa_{d-|\Delta_-^G \setminus \Delta_-^K|}\bigg(
    w^{-1}\cdot(\lambda +\rho_K) - \left(\nu + 2\rho_G - \rho_K\right)
    \ ; \     \Delta^G_+ \setminus \Delta^K_+
    \bigg) \\
    &=& \sum_{w\in W_K}     (-1)^{\ell(w)}  \kappa_{d-|\Delta_-^K \setminus \Delta_-^K|}\bigg(
    w\cdot(\lambda +\rho_K) - \left(\nu + 2\rho_G - \rho_K\right)
    \ ; \     \Delta^G_+ \setminus \Delta^K_+
    \bigg) \qquad \text{by reindexing.}
  \end{eqnarray*}
  Since $T_K$ contains regular elements of $G$, 
  the set $\Delta^G_+ \setminus \Delta^K_+$ lives in a proper half-space,
  so for $d \gg 0$ this vector partition function must vanish.
  Hence we can sum over $d$ (effectively forgetting the
  $\complexes^\times$-action), obtaining the finite $K$-multiplicities
  claimed.
\end{proof}

\junk{
  To draw pictures of examples, it can be helpful to reinterpret the formula
  as follows.
  Let $c(\vec v;\ S) : \naturals^{S} \to T_K^*$ be the affine-linear
  function taking $f \mapsto \vec v + \sum_{\vec s \in S} f(\vec s) \vec s$.
  Then the pushforward $c(\vec v;\ S)_*(1)$ of counting measure on
  $\naturals^S$ is the function $\kappa(\lambda - \vec v;\ S)$,
  whose support splays out from $\vec v$ in the $S$ directions.
  (We won't use this result in this paper.)
  
  \begin{Corollary}\label{cor:cones}
    The $K$-multiplicity function from theorem \ref{thm:Blattnerish} is
    $$ \sum_{w\in W_K}     (-1)^{\ell(w)}\ c\bigg(
    w\cdot\big(\nu + 2\rho_G - \rho_K\big) - \rho_K
    \ ; \quad     w \cdot (\Delta^G_+ \setminus \Delta^K_+)
    \bigg)_* (1)
    $$
  \end{Corollary}

  \subsection{Example: $(G,K) = (GL_3,GL_2\times GL_1)$}
  
  Let $\theta$ be conjugation by $diag(1,1,-1)$. So
  \begin{itemize}
  \item $r$ is an isomorphism
  \item $W_K = S_2 \times S_1 \ \leq S_3$
  \item $\rho_K = (1,-1,0)/2$
  \item $\rho_G = (1,0,-1)$
  \end{itemize}
}

There is an extension of Blattner's conjecture in \cite[p376--377]{KV}
to the necessarily smooth preimages $\calO = \pi_P^{-1}(\calO_{\text{closed}})$
of closed $K$-orbits on $G/P$s.

\subsection{Smooth $K$-orbit closures}

\begin{Theorem}\label{thm:smoothorbit}
  Let $X$ be a smooth $K$-orbit closure on $G/B_G^-$,
  and $\calL$ be the line bundle corresponding to the
  dominant $G$-weight $\nu$.
%  We can identify $X^T$ with a $W_K$-invariant subset of $W_G$,
%  and let $W^X$ denote a set of $W_K$-coset representatives of $X^T$.
  Then the multiplicity of the $K$-irrep $V_\lambda$ in the 
  $\complexes^\times$-weight $d$ part of $\chi(G/B_G^-;\ D_X \otimes \calL)$ is
  \begin{eqnarray*}
    && \sum_{w\in X^T}
    (-1)^{ind(T_w X)} 
    \kappa\bigg(
    \big(\lambda - w\cdot\nu -\rho(T_w X) -2w\cdot\rho_G + \sum_{wts(T_w X)} \mu,
    \quad  d - \codim_Y X\big)\ ; \\ 
    && \qquad\qquad\qquad
    \{ (-\mu,0) : \mu\in wts_+(T_w X) \setminus \Delta^K_+ \} \ \coprod\
    \{ (\mu,1) : \mu\in (w\cdot \Delta_+^G) \setminus wts(T_w X) \}     \bigg)
  \end{eqnarray*}
  where we are identifying $(G/B_G^-)^{T_K} \iso W_G$.
\end{Theorem}

\begin{proof}
  We need argue that we are using the correct $\calD_Y$-module
  with which to compute the irrep (and not e.g. the standard module,
  or something in between).
  Under the correspondence with perverse sheaves, we want to make
  sure that we're using the {\em perverse} extension to the $K$-orbit closure
  of the locally constant sheaf from the $K$-orbit. Since the $K$-orbit
  closure is smooth, the locally constant sheaf on the closure is already
  perverse, and as the perverse extension is unique we have found it.
  Converting back to $\calD_Y$-modules gives us the sheaf of 
  distributions on the orbit closure.

  The rest is easy conversion from the general case to $Y = G/B_G^-$, with
  some special features because $Y$'s tangent bundle is $W_G$-equivariant
  (or rather, $N(T_G)$-equivariant).
\end{proof}

\section{$(\lie{sl}_3,SO(3))$-irreps: 
  proof of theorem \ref{thm:main}}\label{sec:proofofmain}

\newcommand\onehalf{\frac{1}{2}}

We now apply theorem \ref{thm:smoothorbit} to the four $SO(3)$-orbit
closures on $SL(3)/B_{SL(3)}^-$. We index $SL(3)$'s dominant weights
by $\naturals^2$ using the fundamental weights, so $\rho_G = (1,1)$.
But when we index $SO(3)$'s (not $SL(2)$'s) dominant weights by $\naturals$,
we miss the $SL(2)$ weight $\rho_K$, which is at $\onehalf$
where the physicists (used to) like having it. In these coordinates, 
the restriction map $r: T_G^* \onto T_K^*$ is $(a,b) \mapsto a+b$ 
(taking the fundamental irreps of $SL(3)$ to the first, defining, 
irrep of $SO(3)$), and $r(\Delta^G_+) \setminus \Delta^K_+ = \{1,2\}$.

Let $\nu = (a,b) \in T_G^*$ as in the statement of theorem \ref{thm:main}.
The fixed points $Y^{T_K}$ are the coordinate flags, where the 
$SO(3)$-invariant inner product on $\complexes^3$ is the one where
the dual basis to $(e_1,e_2,e_3)$ is $(e_3,e_2,e_1)$. 

\subsection{The closed orbit $\calO_1 \cap \calO_2$}
In this case we can even apply theorem \ref{thm:Blattnerish}, obtaining
\begin{eqnarray*}
&& \sum_{w\in W_K}     (-1)^{\ell(w)}  \kappa\bigg(
(\lambda +\rho_K) - w\cdot\big(\nu + 2\rho_G - \rho_K\big)
\ ; \    w\cdot( \Delta^G_+ \setminus \Delta^K_+)
\bigg) \\
&=& \kappa\big(\lambda + \onehalf - (a+b + 4 - \onehalf)\ ; \   \{1,2\}\ \big) 
\quad-\quad
\kappa\big(\lambda + \onehalf + (a+b + 4 - \onehalf)\ ; \   \{-1,-2\}\ \big) \\
&=& \kappa\big(\lambda - 3 - (a+b)\ ; \     \{1,2\}\ \big) 
\quad-\quad
\kappa\big(\lambda + 4 + (a+b)\ ; \     \{-1,-2\} \ \big) \\
&=& \kappa\big(\lambda - 3 - (a+b)\ ; \     \{1,2\}\ \big) 
\qquad\qquad\qquad \text{since $\lambda,a,b\geq 0$}
\end{eqnarray*}
This vanishes for $\lambda < a+b+3$, and the values for 
$\lambda \geq a+b+3$ are $\lambda \mapsto \lfloor (\lambda-(a+b+3))/2 \rfloor$,
i.e. from $a+b$ they are $0,0,0,1,1,2,2,3,3,4,4,\ldots$
This easily matches the projection of
$$ C_{\calO_1\cap \calO_2} 
:= \{(c,d) \in \naturals^2 : c>a,\ d>b,\ c \equiv a \bmod 2\} 
\quad
\text{from theorem \ref{thm:main}.} $$

W. Schmid pointed out that the normal bundle to the closed orbit 
$\iso \mathbb{CP}^1$ is (by Grothendieck's theorem)
a sum of line bundles, specifically $\calO(-3)\oplus \calO(-3)$.
However Grothendieck's theorem fails $SO(3)$-equivariantly, 
even though it holds $\complexes^\times$-equivariantly \cite[\S 4.1]{KS}; 
rather this bundle has one $SO(3)$-invariant subbundle $\calO(-2)$ 
with quotient $\calO(-4)$. If we replace the normal bundle by
its associated graded bundle, the symmetric algebra of that is again 
a sum of line bundles (controlled by the vector partition function on
$\{1,2\}$), to which we can apply Borel-Weil, obtaining this
vector partition function formula very directly. Then the vanishing
of higher cohomology for the (symmetric algebra of that) associated graded
lets us infer that the normal bundle itself gives us the same $H^0$.

For uniformity with the other orbits to come we redo the calculation
using the formula from theorem \ref{thm:geometry} directly, which computes the
subspace of $\complexes^\times$-weight $d$:

\begin{eqnarray*}
  &&  \sum_{f\in X^T}
  (-1)^{ind(T_f X)} 
  \kappa\bigg(
  \big(\lambda -wt(\calL|_f) -\rho(T_f X) -\sum_{wts(T_f Y / T_f X)} \mu,\quad
  d - \codim_Y X\big)\ ; \\ 
  && \qquad\qquad\qquad
  \{ (-\mu,0) : \mu\in wts_+(T_f X) \setminus \Delta^K_+ \} \ \coprod\
  \{ ( \mu,1) : \mu\in wts(T_f Y/T_f X) \}     \bigg)
\end{eqnarray*}
The $T_K$-fixed points on $\calO_1 \cap \calO_2$ look like
$$
\begin{array}{c|c|c|c|c|c}
f \in \calO_1^{T_K} & wt(\calL|_f) 
& wts(T_f X) & wts(T_f Y)\setminus wts(T_f X) & ind(T_f X) & \rho(T_f X) \\
\hline
(e_1 < e_1 \oplus e_2) & (a,b)    &\{1\}  & \{ 2, 1\} & 0 & 0  \\ %  a+b,b,0
(e_3 < e_2 \oplus e_3) & (-b,-a)  &\{-1\} & \{-2,-1\} & 1 & -1    % 0,b,a+b
\end{array}	
$$
giving two terms
$$
   \kappa\big((\lambda-(a+b)-0-3,d-2);\ \{(1,1),(2,1)\}\big)
\ -\  \kappa\big((\lambda+(a+b)+1+3,d-2);\ \{(-1,1),(-2,1)\}\big) 
$$
the second of which obviously vanishes, the first
simplifying to $\kappa_{d-2}(\lambda-(a+b)-3;\ \{1,2\})$.
Now we sum over $d$ to get $\kappa(\lambda-(a+b)-3;\ \{1,2\})$ as expected.

\subsection{The orbit $\calO_1$}\label{ssec:O1}
Recall $\calO_1 = \{(V_1 < V_2) : rank(V_1) = 0\}$. Its $T_K$-fixed points 
look like
$$
\begin{array}{c|c|c|c|c|c}
f \in \calO_1^{T_K} & wt(\calL|_f) 
& wts(T_f X) & wts(T_f Y)\setminus wts(T_f X) & ind(T_f X) & \rho(T_f X) \\
\hline
(e_1 < e_1 \oplus e_2) & (a,b)    &\{1,1\}   & \{ 2\} & 0 & 0  \\ %  a+b,b,0
(e_1 < e_1 \oplus e_3) & (a+b,-b) &\{-1,1\}  & \{ 2\} & 1 & -1 \\ %  a+b,0,b
(e_3 < e_1 \oplus e_3) & (b,-a-b) &\{-1,1\}  & \{-2\} & 1 & -1 \\ %  b,0,a+b
(e_3 < e_2 \oplus e_3) & (-b,-a)  &\{-1,-1\} & \{-2\} & 2 & -2    % 0,b,a+b
\end{array}	
$$
(where one can figure out the weight on the normal bundle
from the fact that the equation defining $\calO_1$ is quadratic),
giving the terms
$$
\begin{array}{rlcccccc}
  & &\kappa\big(\ (\lambda& - (a+b)& -0&- 2,&d-1);&\ \{(-1,0),\ (2,1)\}\ \big)\\
  &-&\kappa\big(\ (\lambda& - a    & +1&- 2,&d-1);&\ \{(-1,0),\ (2,1)\}\ \big)\\
  &-&\kappa\big(\ (\lambda& + a    & +1&+ 2,&d-1);&\ \{(-1,0),(-2,1)\}\ \big)\\
  &+&\kappa\big(\ (\lambda& + (a+b)& +2&+ 2,&d-1);&\ \{(-1,0),(-2,1)\}\ \big) 
\end{array}
$$
Each of the latter two vector partition functions $\kappa(\vec v;S)$ obviously 
vanishes, since the first coordinate of $\vec v$ is strictly positive, 
but the first coordinate of each $\vec s\in S$ is nonpositive.

In the first two (and the dropped two) terms, we know that the second
vector $(\pm 2,1)$ in the vector partitions must be used $d-1$ times. 
So subtract that off and drop the coordinate:
\begin{eqnarray*}
  & & \kappa(\lambda - (a+b) - 2 -2(d-1));\ \{-1\}) 
  \quad-\quad \kappa(\lambda - a -1 -2(d-1));\ \{-1\}) \\
  &=& [\lambda - a-b \leq 2d] \quad -\quad [\lambda - a + 1 \leq 2d] \\
  &=& [\lambda-a-b \ \leq\ 2d \ <\ \lambda-a+1]
  \qquad\qquad\quad =\quad [\lambda-a-b \ \leq\ 2d \ \leq \lambda-a]
\end{eqnarray*}
As predicted in theorem \ref{thm:geometry}, this is nonnegative for every $d$.
When we sum over $d\in\naturals$, or $d':=2d \in 2\naturals$,
we get the multiplicity of
the $\lambda$ $SO(3)$-irrep inside our $(\lie{sl}(3),SO(3))$-irrep.
To prove theorem \ref{thm:main} in this case, we need to correspond
$ \{ d'\ :d'\in [\lambda-a-b, \lambda-a]\text{ even} \}$ with the
points in $C_{\calO_1} = \{(x,y)\ :\ x > a,\ x \equiv a \bmod 2,\ y \in [0,b] \}$.
The correspondence is 
% lambda-a   |-> (lambda-eps,eps)
% lambda-a-k |-> (lambda-eps-2k,eps+2k)
$$ d' = \lambda-a-k \ \mapsto\ (\lambda-k,\ k). $$
This works since $2|d' \iff \lambda-k \equiv a \bmod 2$, 
and $d' \in [\lambda-a-b,\lambda-a] \iff k \in [0,b]$.

\subsection{Duality and the orbit $\calO_2$}\label{ssec:O2}

The duality automorphism 
$(V_1 < V_2) \stackrel{*}{\mapsto} (V_2^\perp < V_1^\perp)$ of $G/B_G^-$
acts on $K$-orbit closures, $\calO \mapsto \calO^*$,
exchanging $\calO_1,\calO_2$ and fixing the other two.
Pulling back the line bundle switches $(a,b)$ to $(b,a)$, 
giving $SO(3)$-isomorphic representations. 
So the $K$-multiplicity calculation for $\calO_2$ is
easily copied from the one just done for $\calO_1$.
However, the $SO(3)$-multiplicity formul\ae\ are not then visibly the same, 
which we fix with the following.

\begin{Theorem}
  The map $ (c,d) \mapsto
  \begin{cases}
    \qquad (d,\ c)     &\text{if $d \equiv a\bmod 2$} \\
    (d+1,c-1) &\text{if $d \not\equiv a\bmod 2$} 
  \end{cases}
  $
  \quad bijects $C$ with itself, and $C_\calO(x,y)$ with $C_{\calO^*}(y,x)$,
  where $C_\calO(x,y)$ means $C_\calO$ with parameters $(a,b) = (x,y)$. \\
  Moreover, this duality on $C$ preserves the projection $(c,d)\mapsto c+d$.
\end{Theorem}

We leave the very simple details to the reader.
Essentially, the map flips figure \ref{fig:Cpic} upside down;
half the dots land on the illegal brown dashed lines, so we jiggle
them upward to good positions, leaving holes along the bottom edge 
where holes belong.

Now invoke this combinatorial duality to reduce theorem \ref{thm:main} 
for $\calO_2$ to the just-proven result for $\calO_1$.

\subsection{The open orbit}
Now all $3!$ fixed points are in our orbit closure, looking like
$$
\begin{array}{c|c|c|c|c}
f \in \calO_1^{T_K} & wt(\calL|_f) 
& wts(T_f X) & ind(T_f X) & \rho(T_f X) \\
\hline
(e_1 < e_1 \oplus e_2) & (a,b)    &\{1,1,2\}   & 0 &  0 \\ %  a+b,b,0
(e_1 < e_1 \oplus e_3) & (a+b,-b) &\{-1,1,2\}  & 1 & -1 \\ %  a+b,0,b
(e_2 < e_1 \oplus e_2) & (-a,a+b) &\{1,-1,2\}  & 1 & -1 \\ %  b,a+b,0
(e_2 < e_2 \oplus e_3) & (-a-b,a) &\{1,-1,-2\} & 2 & -3 \\ %  0,a+b,b
(e_3 < e_1 \oplus e_3) & (b,-a-b) &\{1,-1,-2\} & 2 & -3 \\ %  b,0,a+b
(e_3 < e_2 \oplus e_3) & (-b,-a)  &\{-1,-1,-2\}& 3 & -4    %  0,b,a+b
\end{array}	
$$
and giving the (essentially Kostant multiplicity) formula
$$
\begin{array}{rlccccl}
  & &\kappa\big(\ (\lambda& -(a+b)& -0,&d);&\ \{(-1,0),(-2,0)\}\ \big)\\
  &-&\kappa\big(\ (\lambda& -a    & +1,&d);&\ \{(-1,0),(-2,0)\}\ \big)\\
  &-&\kappa\big(\ (\lambda& -b    & +1,&d);&\ \{(-1,0),(-2,0)\}\ \big)\\
  &+&\kappa\big(\ (\lambda& +b    & +3,&d);&\ \{(-1,0),(-2,0)\}\ \big)\\
  &+&\kappa\big(\ (\lambda& +a    & +3,&d);&\ \{(-1,0),(-2,0)\}\ \big)\\
  &+&\kappa\big(\ (\lambda& +(a+b)& +4,&d);&\ \{(-1,0),(-2,0)\}\ \big),
\end{array}
$$
nicely symmetric in $a,b$.
Each term vanishes for $d\neq 0$, so take $d=0$ and eliminate the
second coordinate. The last three terms are obviously zero. What remains is
$$
              \kappa(\lambda -(a+b)   ;\ \{-1,-2\}\ )
  \quad-\quad \kappa(\lambda -a     +1;\ \{-1,-2\}\ )
  \quad-\quad \kappa(\lambda -b     +1;\ \{-1,-2\}\ )
$$
Each term contributes a multiplicity function like 
``$\ldots,3,3,2,2,1,\underline 1,0,\ldots$''
where the $\underline 1$ is at $a+b$, $a-1$, or $b-1$.

Consider the set $D = \{(c,d) \in -\naturals^2 : c \equiv 0 \bmod 2\}$,
whose projection $(c,d) \mapsto c+d \in -\naturals$ also gives this
"$\ldots,3,3,2,2,1,\underline 1,0,\ldots$" multiplicity function 
on $-\naturals$, with the $\underline 1$ at $0$. 
Then the translate $(a,b)+D$ contains $C_{G/B_G^-}$; in figure \ref{fig:Cpic},
$(a,b)+D$ is the quadrant to the West of $C_{G/B_G^-}$'s East corner.
If we take the difference $(a,b)+D \setminus (a,-1)+D$
of translates, we get a half-infinite strip with the same Northeast side
as $C_{G/B_G^-}$, and this difference projects to the multiplicity function
$\kappa(\lambda -(a+b)   ;\ \{-1,-2\}\ )
- \kappa(\lambda -a     +1;\ \{-1,-2\}\ )$. 

It remains to subtract off $\kappa(\lambda -b     +1;\ \{-1,-2\}\ )$
using another translate of $D$. There are two cases. 

\newcommand\bfchi[1]{{{\bf\chi}\big( #1 \big)}}
The easy case is $a$ odd, where $(a,b)+D$ contains $(-1,b)+D$,
to the Southwest of $C_{G/B_G^-}$.
The subtlety is that $(a,b)+D \setminus (a,-1)+D$ does not
contain $(-1,b)+D$,
but the missing points $(c,d)$ all have $c+d < 0$, so they don't
affect the point-count in the region we care about.
Put another way, if we use $\bfchi{S}$ to indicate the characteristic
function of the set $S$, then
$$ + \bfchi{(a,b)+D} - \bfchi{(a,-1)+D} - \bfchi{(-1,b)+D} 
= + \bfchi{C_{\calO_1\cap \calO_2}} \quad\text{on $\{(c,d) : c+d\geq 0\}$} $$
so the two sides induce the same counts when projected to $\naturals$.

The tricky case is $a$ even, where we attempt to subtract off $(-2,b+1)+D$
from $(a,b)+D \setminus (a,-1)+D$. Even on the $c+d\geq 0$ side,
we don't quite get $C_{G/B_G^-}$, but instead
$$ 
  + \ \bfchi{ (c,d) \in \naturals^2 : c \equiv a \bmod 2, c\leq a, d\leq b}
\ - \ \bfchi{ (-2-2n,b+1) : n \geq 0 } 
$$
or equivalently
$$ 
  + \ \bfchi{C_{G/B_G^-}}
  + \ \bfchi{ (0,d) \in \naturals^2 : d\leq b, d\equiv b-1 \bmod 2 }
\ - \ \bfchi{ (-2-2n,b+1) : n \geq 0 } 
$$
(again, this computation is only correct over $c+d\geq 0$). If we project
the latter two terms to $\naturals$ using $(c,d)\mapsto c+d$, they become
$$ 
  + \ \bfchi{ d \in \naturals : d\leq b, d\equiv b-1 \bmod 2 }
\ - \ \bfchi{ b-1-2n : n \geq 0 } 
$$
and visibly cancel.
Hence in either the $a$ even or odd cases, the projection of $C_{G/B_G^-}$
correctly computes the $SO(3)$-multiplicities in the
finite-dimensional $SL(3)$-irrep $V_{(a,b)}$.

Note that the $SO(3)$-multiplicities in $V_{(a,b)}$ and $V_{(b,a)}$
are the same, which is not manifest from counting points in $C_{G/B_G^-}$;
this symmetry can be seen combinatorially using the bijection
from \S \ref{ssec:O2}.

In the case $a,b$ both even, the $SO(3)$-multiplicities start out 
$1,0,2,1,\ldots$, i.e. they include every $SO(3)$-irrep in $[0,a+b]$ except
for a hole at $1$ (the $3$-dim representation). Such holes never occur
in weight multiplicities or $GL(n)$ tensor product multiplicities \cite{KT}.
Another study of such holes in the Blattner situation appeared in \cite{H}.

\end{document}